\pdfoutput=1

\documentclass{amsart}
\usepackage{graphicx}
\usepackage{amssymb, amsmath, amsthm,amscd,epsfig,graphicx}
\usepackage{verbatim}
\usepackage{parskip}
\usepackage[english]{babel}
\usepackage[utf8x]{inputenc}
\usepackage{dcpic,pictexwd}

\begin{document}
\newcommand{\SymL}{\mathbf{L}^{\bullet}}
\newcommand{\QuadL}{\mathbf{L}_{\bullet}}

\newcommand{\holim}[1][]{\mathop{\mathrm{holim}}\limits_{#1}}
\newcommand{\hocolim}[1][]{\mathop{\mathrm{hocolim}}\limits_{#1}}

\newcommand{\holimtwo}[2]{\mathop{\mathrm{holim}}_{\substack{{#1}\\{#2}}}}
\newcommand{\limtwo}[2]{\mathop{\mathrm{lim}}_{\substack{{#1}\\{#2}}}}
\newcommand{\spectrum}[1]{\mathbf{#1}}
\newcommand{\catfont}[1]{\mathrm{#1}}
\newcommand{\Zn}[1][n]{\mathbb{Z}[x]/x^{#1}}
\newcommand{\Hom}{\mathrm{Hom}}
\newcommand{\cell}{\mathrm{cell}}
\newcommand{\Mor}{\mathcal{M}\mathrm{or}}
\newcommand{\SWh}{\odot} 

\newcommand{\coker}{\mathop{\mathrm{coker}}}
\newcommand{\Nil}{{\mathrm{Nil}}}
\newcommand{\End}{{\mathrm{End}}}
\newcommand{\A}{\mathrm{A}}

\newcommand{\Obj}{\mathrm{Obj}}
\newcommand{\cSP}{\mathrm{Spaces}^c_*}
\newcommand{\Set}{\mathrm{Set}}
\newcommand{\coRarrow}{\rightarrowtail}
\newcommand{\wRarrow}{\xrightarrow{\simeq}}
\newcommand{\cowRarrow}{\stackrel{\simeq}{\rightarrowtail}}
\newcommand{\quotRarrow}{\twoheadrightarrow}
\newcommand{\dtilde}[1]{\widetilde{\widetilde{#1}}}
\newcommand{\stilde}[1]{\widetilde{#1}}
\newcommand{\Cone}{\mathrm{Cone}}

\newcommand{\TT}{\mathbb{T}}
\newcommand{\ZZ}{\mathbb{Z}}
\newcommand{\CC}{\mathbb{C}}
\newcommand{\NN}{\mathbb{N}}
\newcommand{\DD}{\mathbb{D}}

\newcommand{\xrightleftarrows}[2]{\mathrel{\mathop{\rightleftarrows}^{#1}_{\mathrm{#2}}}}
\newcommand{\xleftrightarrows}[2]{\mathrel{\mathop{\leftrightarrows}^{#1}_{\mathrm{#2}}}}

\newcommand{\Proj}{\mathbb{P}}
\newcommand{\hofib}{\mathrm{hofib}}
\newcommand{\hocofib}{\mathrm{hocofib}}
\newcommand{\id}{\mathrm{id}}
\newcommand{\cyl}{\mathrm{cyl}}
\newcommand{\Map}{\mathop{Map}}

\title{The $K$-Theory of Endomorphisms of Spaces}
\author{Filipp Levikov}
\email{filipp.levikov@fu-berlin.de}%
\address{Freie Universität Berlin\\Mathematisches Institut\\Arnimallee 7\\14195 Berlin\\Germany}%
\thanks{The work carried out in the article was partially supported by the Humboldt Award of Michael Weiss.}


\begin{abstract}
We prove a non-linear version of a theorem of Grayson which is an analogue of the Fundamental Theorem of Algebraic $K$-theory 
and identify the $K$-theory of the endomorphism category over a space $X$ in terms of reduced $K$-theory of a certain localisation of the category of $\NN$-spaces over $X$. In particular we generalise the result of Klein and Williams. 
\end{abstract}
\date{\today}
\maketitle

\theoremstyle{definition}
\newtheorem{definition}{Definition}[section]
\newtheorem{example}[definition]{Example}
\newtheorem{remark}[definition]{Remark}
\newtheorem{corollary}[definition]{Corollary}
\newtheorem{theorem}[definition]{Theorem}
\theoremstyle{plain}
\newtheorem{lemma}[definition]{Lemma}
\newtheorem{oplemma}[definition]{! Lemma}
\newtheorem{proposition}[definition]{Proposition}

\section{Introduction}
The fundamental theorem of algebraic $K$-theory of rings consists of two statements. The Bass-Heller-Swan splitting describes the $K$-theory of the Laurent polynomial extension of a ring $R$
\[
K_i(R)[t,t^{-1}]\cong K_i(R)\oplus K_{i-1}(R)\oplus NK_i(R)\oplus NK_i(R) 
\] 
with $NK_i(R)$ given by the kernel of the projection $K(R[t])\rightarrow K(R)$. The second statement gives a $K$-theoretic interpretation of $NK_i(R)$ in terms of $K$-theory of the exact category of nilpotent endomorphisms of projectives over $R$
\[
NK_i(R)\cong\Nil_{i-1}(R).
\]
More generally one can consider the $K$-theory of the endomorphism category \linebreak $\End(R)$ of projectives over $R$. The investigation of the corresponding $K_0$ goes back to Almkvist (\cite{Almk74}). Grayson (\cite{Gra77}) interpeted it as $K$-theory of a certain localisation of the polynomial ring $R[t]$.
Grayson also realised that the Nil-description and the End-description can be given at once by introducing a parameter $S$ consisting of a multiplicative set of centric monic polynomials in $R[t]$ (containing $t$):
\begin{equation}\label{Eq:Grayson's Theorem}
 \tilde{K}_i(\End_S(R))\cong \tilde{K}_{i+1}\big(\tilde{S}^{-1}R[t]\big). 
\end{equation}
In \cite{FTI}, the authors prove a Bass-Heller-Swan decomposition for \linebreak Waldhausen's algebraic $K$-theory of spaces, alias $A$-theory. Together with the identification of the nil-terms in \cite{FTIII} this has every right to be called the \textit{Fundamental theorem of algebraic $K$-theory of spaces}.   Recently, there has been interest in $K$-theory of parametrised endomorphisms (\cite{Bet05}, \cite{LMc12}) and its connection to topological cyclic homology (\cite{Hesselholt96}, \cite{BS05}). The preprint \cite{BGT} offers a more conceptual point of view and deals with the representability question. In this paper we have a much more modest goal of providing the $A$-theoretic analogue of \eqref{Eq:Grayson's Theorem}. This is Theorem \ref{Main Theorem}. In a sequel to this paper we will bring involutions into the picture and deal with the question of equivariance.
\section*{acknowledgements}  
The result obtained in this article is based on a part of my PhD thesis written at the University of Aberdeen under supervision of Michael Weiss. I am very grateful for his ongoing support. 
\section{Definitions and Main Theorem}
\subsection{}Let $M_\bullet$ be a simplicial monoid with realisation $M=|M_\bullet|$. Consider the category $\mathbb{T}(M)$ of based spaces with a (left) action by $M$ and morphisms given by $M$-maps, i.e. maps commuting with the $M$-action. Define an $M$-cell of dimension $n$ to be 
\[
D^n\times M.
\]
Given an object $Z\in\TT(M)$ and an equivariant map $\alpha:S^{n-1}\times M\rightarrow Z$ define a {cell attachment} by the pushout
\[
Z\cup_\alpha(D^n\times M).
\]
A map $Y\rightarrow Z$ in $\TT(M)$ will be called a \textit{cofibration} if $Z$ is a retract of cell attachments on $Y$. It will be called a \textit{weak equivalence} if it is a weak homotopy equivalence in the ordinary sense. Denote by $\CC(M)$ the full subcategory of $\TT(M)$ consisting of cofibrant objects, i.e. made out of $\ast$ by successive cell attachments. There are several notions of finiteness. An object of $\CC(M)$ is called \textit{finite} if it is built out of finitely many cells. It is called \textit{homotopy finite} if there exists a weak equivalence to a finite object. It is called \textit{finitely dominated} if it is a retract of a homotopy finite objcet. There are inclusions of categories
\[
\CC_f(M)\subset\CC_{hf}(M)\subset\CC_{fd}(M)\subset\CC(M)\subset\TT(M) 
\]
where the indices read as finite, homotopy finite and finitely dominated respectively. At some point we will also need the corresponding category of \textit{stably finitely dominated} objects. This will be indicated by the index $sfd$.

The above subcategories of $\CC(M)$ along with the specified classes of cofibrations and weak equivalences form Waldhausen categories and we can apply\linebreak Waldhausen's S-construction to the respective subcategories of weak equivalences. These will be denoted by a prefix $w$. This leads to the definition of $A$-theory in the sense  of \cite{Wald85} for various values of $?$
\[
A^?(\ast,M)=\Omega|wS_\bullet\CC_?(M)|. 
\]
For a based connected space $X$, one sets $M_\bullet$ to be its Kan loop group $G_\bullet=\Omega_\bullet(X)$. In this case the above definition is one definition of $A^?(X)$ - of the algebraic $K$-theory of $X$. 

\medskip

\subsection{}
Let $Y$ be an $M$-space. The stable homotopy classes of $M$-maps $\{Y,Y\}=\{Y,Y\}_M$ form a group under addition which is both left and right distributive with respect to composition of maps making $\{Y,Y\}$ into a ring. The $M$-action on $Y$ induces a map from $M$ to the (stable) endomorphisms of $Y$ which respect this action. Passing to path components and taking the centre $\mathcal{Z}$ we get a map
\[
\mathcal{Z}\pi_0(M)\rightarrow\{Y,Y\}_M 
\]
which turns $\{Y,Y\}_M$ into an algebra over $\mathbb{Z}[\mathcal{Z}\pi_0(M)]$.
As a consequence we can evaluate polynomials with coefficients in $\ZZ[\mathcal{Z}\pi_0(M)]$ on $\{Y,Y\}_M$. We want to introduce a parameter for the study of (stable) endomorphisms of $M$-spaces. Consider a multiplicative set $S$ consisting of monic polynomials in the polynomial ring $\ZZ[\mathcal{Z}\pi_0(M)][x]$ containing the multiplicative set $\{x^n,n\geq0\}$. 
Define the endomorphism category (with respect to $S$) $\End_?^S(M)=\End_?^S(\ast,M)$ as follows. The objects are pairs $(Y,f)$ with $Y$ an object in $\CC_?(M)$ and $f$ an endomorphism of $Y$ with the property that there exists a polynomial $h\in S$ (stably) homotopy annihilating $f$, i.e. such that $h[f]$ vanishes in the $\ZZ[\mathcal{Z}\pi_0(M)]$-algebra of stable homotopy classes of $M$-self-maps $\{Y,Y\}_M$. Morphisms are given by commutative squares

\[
\begindc{0}[40]
\obj(1,2)[1]{$Y$} 
\obj(2,2)[2]{$Y$}
\obj(1,1)[3]{$Y'$}
\obj(2,1)[4]{$Y'$}
\mor{1}{2}{$f$}
\mor{2}{4}{}[1,0]
\mor{1}{3}{}[1,0]
\mor{3}{4}{$f'$}
\enddc 
\]

Using the forgetful functor 
\[
\mathrm{p}:\End_?^S(\ast,M)\rightarrow \CC_?(M) 
\]
we can pull back the notions of cofibrations and weak equivalences to the endomorphism category.
\begin{remark}
The condition on $f$ being homotopy annihilated by a polynomial in $S$ requires us to make sense out of addition in $[Y,Y]$. For this reason we pass to homotopy classes of stable maps. It might be easier to work with finite (resp. finitely dominated) spectra straight away. We prefer, however, to stay in the unstable category to make the connection of the present work to \cite{FTI},\cite{FTII},\cite{FTIII} more visible. Since we are concerned with $K$-theoretic considerations and suspension induces an equivalence on $K$-theory there is no difference. 

Note that a polynomial $h\in\ZZ[\mathcal{Z}\pi_0(M)]$ can be evaluated at an endomorphism $f$ (after suspension), however the result $h(f)$ depends on choices and only the class $[h(f)]$ does not. We will nevertheless sometimes use the notation $h(f)$ for any representative.
\end{remark}

\begin{lemma}
 The category $\End_?^S(\ast,M)$ is a Waldhausen category.
 \end{lemma}
\begin{proof}
 One can proceed almost verbatim as in the proof of \cite[Lemma 2.1]{FTIII}. The only surprising fact is that endomorphisms with the property of being stably annihilated by $S$ are closed under cobase change. Given a diagram in $\End_?^S(M)$ 
 \[
 (B,f_1)\leftarrow(A,f_0)\coRarrow(C,f_2)
 \]
consider its pushout $(B\cup_AC,f)$ with $f$ defined by $f_1\cup_{f_0}f_2$. For $i=0,1,2$, there exist polynomials $h_i\in S$ stably homotopy annihilating $f_i$. Choose a polynomial $h$ homotopy annihilating $f_i$ for all $i$, e.g. $h=h_0h_1h_2$. To work with unstable homotopy classess and to simplify notation we assume that $A,B,C$ are already $n$-fold suspensions, with $n$ large enough. The cofibration sequence
\[
B\vee C\xrightarrow{j}B\cup_AC\xrightarrow{\delta}\Sigma A 
\]
induces an exact sequence 
\[
[\Sigma A,B\cup_AC]\xrightarrow{\delta^\ast}[B\cup_AC,B\cup_AC]\xrightarrow{j^\ast}[B\vee C,B\cup_AC]
\]
and
\[
j^\ast(h[f])=[h(f)\circ j]=[h(f_1)\vee h(f_2)]=0 
\]
implies that there is a class $\gamma\in[\Sigma A,B\cup_AC]$ such that $h[f]=\delta^\ast(\gamma)=\gamma\circ[\delta]$ and consequently 
\[
h[f]^2=\gamma\circ[\delta]\circ\gamma\circ[\delta].
\]
Since $[\delta]\circ\gamma\circ[\delta]=[\delta]\circ h[f]$ is just $[\Sigma h(f_0)]\circ[\delta]$ and $[\Sigma h(f_0)]$ is trivial we conclude that $h[f]^2$ is zero, i.e. that $h^2$ stably homotopy annihilates $f$.
\end{proof}
\begin{example}\label{Examples}
\quad\\\begin{enumerate}
\item Let $X$ be a based connected space and $G$ the realisation of its Kan loop group. 
Let $S$ be the multiplicative set given by monomials $x^n$. In this case the (finitely dominated) End-category coincides with the Nil-category introduced in \cite{FTIII} 
\[
\End_{fd}^S(\ast,G)=\Nil_{fd}(\ast,G)=\Nil_{fd}(X). 
\]
\item Assume $\pi_0(M)$ is commutative and let $S$ consist of all monic polynomials in $\ZZ[\pi_0M][x]$. In this case the endomorphism category consists of \textit{all} endomorphisms
\[
\End_?^S(\ast,M)=\End_?(\ast,M).
\]
This can be justified as follows. For a finite $M$-space $Y$ we can consider its chain complex of free $\ZZ[\pi_0M]$-modules. The induced chain map $f_\ast$ possesses in every degree $k$ a characteristic polynomial $h_k$ \footnote{Characteristic polynomials can be defined for endomorphisms of projective $R$-modules for any commutative ring $R$ (cf. \cite[p. 631]{Bass68} and \cite{Alm73}).} and the Cayley-Hamilton theorem implies that $h_k(f_k)=0$. Thus a finite product $h$ of the $h_k$'s annihilates $f_\ast$. Since $h(f)$ induces an acyclic map of $\ZZ[\pi_0M]$-module chain complexes it is stably nullhomotopic through an $M$-homotopy. In the homotopy finite case one can pass to the corresponding finite situation. 

Assume now that $Y$ is finitely dominated with $f$ any endomorphism. There is a retraction from a homotopy finite space $\bar{Y}$
\[
Y\xleftrightarrows{\quad r\quad}{\quad i\quad}\bar{Y} 
\]
and we can consider the endomorphism $\bar{f}=i\circ f\circ r :\bar{Y}\rightarrow\bar{Y}$. Because of the above there exists a polynomial $h$ with the property that $h[\bar{f} ]=0$ in $\{\bar{Y},\bar{Y}\}_M$. This however immediatly implies that $h[f]=0$ in $\{Y,Y\}_M$. 
\end{enumerate}
\end{example}
\subsection{}
Let $L$ denote one of the monoids: the non-positive integers $\NN_-$, the non-negative integers $\NN_+$ or all integers $\ZZ$. Write $t^{-1},t$ or $t$ for the respective generators. 

An object $Y\in\mathbb{C}_?(M\times L)$ will be called $T$-contractible if there exists a polynomial $h$ in $T$ such that the telescope
\[
\hocolim (\cdots\xrightarrow{h(t)}Y\xrightarrow{h(t)}Y\xrightarrow{h(t)}\cdots)
\]
respectively 
\[
\hocolim(\cdots\xrightarrow{h(t^{-1})}Y\xrightarrow{h(t)^{-1}}Y\xrightarrow{h(t)^{-1}}\cdots)
\]
is stably contractible in $\TT(M)$. Note that in general the telescope is \textit{not} in \linebreak$\CC_{fd}(M\times L)$, however it is still a cofibrant $M$-space in $\TT(M)$. A map $f:Y\rightarrow Z$ in $\CC_?(M\times L)$ will be called a \textit{$T$-equivalence} if its homotopy fibre
is $T$-contractible. Since this is a stable condition, it is equivalent to the homotopy cofibre $\hocofib(f)$ being $T$-contractible. 
Let $M$ and $S$ be as before. The \textit{reverse polynomial} $\tilde{h}$  of a polynomial $h\in S$ of degree $n$ is given by $\tilde{h}(x)=h(x^{-1})\cdot x^n$. We will denote the set of all reverse polynomials of polynomials in $S$ by $\tilde{S}$. Write $\mathcal{R}_?^S(M)$ for the following category. The underlying category with cofibrations is given by $\CC_?(M\times\NN_+)$ and the weak equivalences are given by $\tilde{S}$-equivalences. 

\begin{lemma}
 This is a Waldhausen category.
\end{lemma}
There is an exact projection functor
\[
\mathcal{R}_?^S(M)\rightarrow A^?(\ast,M) 
\]
mapping an object $X$ to the quotient $X/\NN_+$ under the action by $t$. This induces a map on $K$-theory and we denote by $E^S_?(M)=E^S_?A(\ast,M)$ the corresponding homotopy fibre. The homotopy fibre of

\[
K(\End^S_?(\ast,M))\rightarrow A^?(\ast,M) 
\]
will be denoted by $\tilde{K}(\End^S_?(\ast,M))$. We aim for the following theorem

\begin{theorem}\label{Main Theorem}
 For $?=f,hf,fd\,$ there is a natural homotopy equivalence of spectra
 \[
\tilde{K}(\End^S_{?}(\ast,M)) \simeq\Omega E_{?}^SA(\ast,M).
 \]
\end{theorem}
\begin{remark} 
We will prove the theorem for ?=fd which is sufficient. In fact, for all values of ? the spectra $\tilde{K}(\End^S_{?}(\ast,M))$ resp. $E_{?}^SA(\ast,M)$ are equivalent by a cofinality argument.
\end{remark}
We want to compare this theorem to two results which exist in the literature. Denote by $A$ the group ring $\ZZ[\pi_0M]$ and let $S$ be as before. Write $\mathbf{End_A^S}$ for the exact category of pairs $(P,f)$ with $P$ a f.g. projective $A$-module and $f$ an endomorphism of $P$ such that there exists a $g$ in $S$ with the property $g(f)=0$. Let $\End_i^SA$ be the kernel of the map on $K$-groups induced by projection
\[
K_i\mathbf{End_A^S}\rightarrow K_iA.
\]
Similarly let $EK_i^S$ be the cokernel of the canonical map
\[
K_iA\rightarrow K_i(\tilde{S}^{-1}A[x]).  
\]
In \cite{Gra77} Grayson obtains
\begin{theorem}\label{Grayson}
There is a natural isomorphism of groups
\[
\End^S_{i-1}A\cong EK_i^SA. 
\]
\end{theorem}
This can now be seen as a corollary to Theorem \ref{Main Theorem} by applying the linearisation functor on both sides of the homotopy equivalence. 

Turn now to the example \ref{Examples} (1). In this case $\mathcal{R}^S_{fd}(X)=\mathcal{R}^S_{fd}(\ast,G)$ can be naturally identified with the nil-term $N_+A^{fd}(X)$ (cf. \ref{Remark:Nil-term}) introduced in \cite{FTI}. In \cite{FTIII} Klein and Williams prove
\begin{theorem}\label{Nil-case}
There is a natural homotopy equivalence of spectra
\[
 \tilde{K}(\Nil_{fd}(X)\simeq \Omega N_+A^{fd}(X).
\]
\end{theorem}
Theorem \ref{Main Theorem} should be seen as a generalisation of the latter. 
\begin{remark}
The category $\mathcal{R}^S_{?}(X)$ might appear as an ad-hoc and unmotivated defi-nition. However, there are two obvious reasons, which in our view provide enough motivation. On the one hand, the $K$-theory of $\mathcal{R}^S_{?}(X)$ linearises to the $K$-theory of Grayson's localisation $\tilde{S}^{-1}\ZZ[\pi_1 X][t]$ (via the $K$-theory of the corresponding category of chain complexes) and since our proof is analogous to Grayspn's original proof, the equivalence of  \ref{Main Theorem} linearises to that of \ref{Grayson}. On the other hand  $\mathcal{R}^S_{?}(X)$ is defined along the same lines as $N_+A^{?}(X)$ and looks like the right generalisation of it. (See also Lemma \ref{Lemma: localisation of D(MxN) is R^S}).

The question, however, remains in what sense  $\mathcal{R}^S_{?}(X)$ \textit{is} a localisation and we respond to the referee's request of placing this into the general context of localisation. If we worked in the framework of higher categories in the sense of Lurie for example, the category $\CC_{?}(M)$ would correspond to the $E_1$-ring spectrum  $\spectrum{S}[M]$, meaning that they give the same (possibly apart from $K_0$) $K$-theory - essentially by the approximation theorem. In this analogy our category  $\CC_{?}(M\times\NN_+)$ would correspond to the $E_1$-ring spectrum $\spectrum{S}[M][t]=\spectrum{S}[M]\wedge \NN_+$ (where we simplify notation by writing $\NN_+$ instead of $(\NN_+)_+$ for the non-negative integers with a disjoint base point). In the above, we are given a mutlplicative subset 
\[
\tilde{S}\subset\ZZ[\mathcal{Z}\pi_0(M)][t]\subset\ZZ[\pi_0(M)][t]\subset\pi_0(\spectrum{S}[M]\wedge\NN_+),
\]
consisting in particular of homogeneous elements in the centre of $\pi_*\big(\spectrum{S}[M]\wedge\NN_+\big)$. By \cite[7.2.4]{Lu15} there exists a ``localisation'' $(\spectrum{S}[M]\wedge\NN_+)[\tilde{S}^{-1}]$ and its $K$-theory is equivalent to $K(\mathcal{R}^S_?(X))$. The latter can be shown by comparing the localisation sequence in $K$-theory (cf. \cite[11.15]{Bar15} or \cite[2.8]{ABG15}) with the one arising in our context: the homotopy fibre of $K(\CC_{?}(M\times\NN_+)\rightarrow K(\mathcal{R}_?^S(X))$ is equivalent (potentially up to $K_0$) to the 
$K$-theory of $\tilde{S}$-nilpotent modules over $\spectrum{S}[M]\wedge\NN_+$ - this can be seen by using the approximation theorem again.

In other frameworks of (structured) ring spectra, e.g. of $S$-algebras in the sense of \cite{EKMM97}, things seem to be less clear since $M$ and thus $\spectrum{S}[M]\wedge \NN_+$ will rarely be commutative. There is no guarantee that something like $(\spectrum{S}[M]\wedge\NN_+)[\tilde{S}^{-1}]$ exists. However we can make sure that the corresponding module category exists by employing derived Cohn localisation (cf. \cite{Dw06}). In this situation it is better interpreted as left Bousfield localisation of the ambient model category (cf. \cite{Hi03}). 

Abbreviate by $\spectrum{E}$ the $S$-algebra $\spectrum{S}[M]\wedge\NN_+$ and consider the subset $\tilde{\mathcal{S}}\subset\Map_\spectrum{S}(\spectrum{S},\spectrum{E})$ of maps projecting to $\tilde{S}$. This gives rise to a set $\mathcal{C}_{\tilde{\mathcal{S}}}$ of self maps of $\spectrum{E}$.  Let $L_{\mathcal{C}_{\tilde{\mathcal{S}}}}\mathcal{M}od(E)$ be the left Bousfiled localisation of the category $\mathcal{M}od(E)$ of $\spectrum{E}$-modules with respect to $\mathcal{C}_{\tilde{\mathcal{S}}}$. The subcategory of homotopy finite (resp. homotopy finitely dominated) cell objects in $L_{\mathcal{C}_{\tilde{\mathcal{S}}}}\mathcal{M}od(E)$ is a Waldhausen category and its $K$-theory is equivalent to $K(\mathcal{R}^S_?(X))$. This can be shown by exhibiting $\mathcal{R}^S_?(X)$ as a Waldhausen subcategory of homotopy finite (resp. homotopy finitely dominated) cell objects of the left Bousfiled localisation of $\TT(M\times \NN_+)$ with respect to a suitable class of maps lifting $\tilde{S}$.

\end{remark}
\section{The projective line}
\subsection{}
We are going to employ the projective line category which was used throughout the series \cite{FTI},\cite{FTII},\cite{FTIII}. In fact since we are going to follow very closely the original proof of \ref{Nil-case}, we present the essential definitions.

Given an object $Y$ in $\CC_{fd}(M\times \NN_+)$ resp. in $\CC_{fd}(M\times \NN_-)$ we can associate to it an obejct $Y(t^{-1})$ resp. $Y(t)$ in $\CC_{fd}(M\times\ZZ)$ which is given by the mapping telescope construction, i.e. as the colimit of
\[
\cdots\xrightarrow{t^{\pm 1}}Y\xrightarrow{t^{\pm 1}}Y\xrightarrow{t^{\pm 1}}\cdots.
\]
Define the category $\DD_{fd}(M\times \ZZ)$ with objects given by diagrams
\[
\overline{Y}=(Y_-\xrightarrow{\quad a_-\quad } Y\xleftarrow{\quad a_+\quad} Y_+) 
\]
with $Y_\pm\in\CC_{fd}(M\times\NN_\pm)$, $Y\in\CC_{fd}(M\times\ZZ)$ and $a_-,a_+$ cofibrations. A morphism $\overline{Y}\rightarrow\overline{Z}$ in $\DD_{fd}(M\times\ZZ)$ consists of componentwise morphisms making the obvious diagram commute. It will be called a \textit{cofibration} if all components as well as the induced maps
\[
Y\cup_{Y_-(t)}Z_-(t)\rightarrow Z \qquad\textnormal{and}\qquad Y\cup_{Y_+(t^{-1})}Z_+(t^{-1})\rightarrow Z
\]
are cofibrations. Define $\DD_{fd}(M\times \NN_\mp)$ to be the full subcategory of $\DD_{fd}(M\times \ZZ)$ consisting of objects satisfying the property that
\[
a_\mp(t^{\pm 1}):Y_{\mp}(t^{\pm 1})\longrightarrow Y 
\]
is a weak equivalence. The \textit{projective line} category $\Proj_{fd}(M)$ is defined as the full subcategory of $\DD_{fd}(M\times \ZZ)$ consisting of objects satisfying that \textit{both} maps
\[
a_-(t):Y_-(t)\longrightarrow Y \qquad\textnormal{and}\qquad a_+(t^{-1}):Y_+(t^{-1})\longrightarrow Y
\]
are weak equivalences.
\begin{remark}\label{Remark:Nil-term}
 The $K$-theory of the category $\mathbb{D}_{fd}(G\times\NN_+)$ can be seen as $K$-theory of the (polynomial) brave new ring $\Sigma^\infty G_+[t]$. Thus the homotopy fibre of
 \[
 K(\mathbb{D}_{fd}(G\times\NN_+))\rightarrow K(\mathbb{C}_{fd}(G)) 
 \]
is a natural candidate for the $A$-theoretic nil-term $N_+A^{fd}(X)$ with $G=|\Omega_\bullet X|$.
\end{remark}

\subsection{Auxiliary results}
The following is the non-linear version of Quillen's decomposition of the $K$-theory of the projective line.
\begin{proposition}\label{Proposition: Decomposition of K-theory of the projective line}[\cite[6.7]{FTI}]
There is a homotopy fibration sequence
\[
K(\CC_{fd}(M))\xrightarrow{\psi_{-1}} K(\Proj_{fd}(M))\xrightarrow{\Gamma} K(\CC_{fd}(M)).
\]
\end{proposition}
To explain the maps we need some more definitions. These will also be important for the proof of the main theorem.
\begin{definition}\label{Definition: Global section functor}[(cf. e.g. \cite[5.1]{FTI})]
For $\overline{Y}=(Y_-,Y,Y_+)\in\Proj_{fd}(M)$ define an object $\Gamma(\overline{Y})$ in $\CC(M)$ as
\[
cY_-\cup_{Y_-}Y\cup_{Y_+}cY_+
\]
where $cY_\pm$ denote the cone $Y_\pm\wedge I$. This is a functor to $\CC_{sfd}(M)$ (cf. \cite[5.2]{FTI}), called the \textit{global sections} functor.
\end{definition}
\begin{definition}\label{Definition: Extension by scalars, twist and canonical sheaf}[(cf. e.g. \cite[5.3,5.4]{FTI})]
The \textit{extension by scalars} functors
\[
\CC_{fd}(M)\rightarrow\CC_{fd}(M\times\NN_+), \quad\CC_{fd}(M)\rightarrow\CC_{fd}(M\times\NN_-), \quad \CC_{fd}(M)\rightarrow\CC_{fd}(M\times\ZZ)
\]
take an $M$-space $Y$ to the $(M\times\NN_\pm)$-space $(\NN_\pm)_+\wedge Y$ resp. the $(M\times \ZZ)$-space $\ZZ_+\wedge Y$.\\

For an integer $n$ the \textit{twist} functor 
\[
\theta_n:\Proj_{fd}(M)\rightarrow\Proj_{fd}(M)
\]
takes $\overline{Y}=(Y_-,Y,Y_+)$ to
\[
Y_-\xrightarrow{t^n\circ a_-}Y\xleftarrow{\quad a_+\quad}Y_+
\]
or putting $\theta_n(a_-)=t^n\circ a_-$ to 
\[
\theta_n(\overline{Y})=(Y_-,Y,Y_+,\theta_n(a_-),a_+).
\] 
The functor
\[
\psi_0:\CC_{fd}(M)\rightarrow\Proj_{fd}(M)
\]
is defined by mapping an $M$-space $Y$ to its \textit{canonical sheaf} given by by applying the extension by scalars functor componentwise
\[
(\mathbb{N}_-)_+\wedge Y\hookrightarrow\mathbb{Z}_+\wedge Y\hookleftarrow(\mathbb{N}_+)_+\wedge Y.
\]
For an integer $n$ the functor 
\[
\psi_n:\CC_{fd}(M)\rightarrow\Proj_{fd}(M)
\]
takes $Y$ to its \textit{canonical sheaf twisted by n} $\theta_n\circ\psi_0(\overline{Y})$.
\end{definition}

\section{The Proof}
\subsection{}
Consider the category $\mathbb{D}_{fd}(M\times\mathbb{N}_+)$ of 2.1. Following \cite{FTI}, we make it into a Waldhausen category by declaring a map $(f_-,f,f_+)$ a weak equivalence if and only if $f_+$ is. As before denote the class of weak equivalences by $h$. We want also to introduce a coarser notion of weak equivalences $h_S$ given by those triples $(f_-,f,f_+)$ for which $f_+$ is an $\tilde{S}-$equivalence. Denote the corresponding class of weak equivalences by $h_S$. 
\begin{lemma}\label{Lemma: localisation of D(MxN) is R^S}
There is a natural (weak) homotopy equivalence
\[
K(h_S\mathbb{D}_{fd}(M\times \mathbb{N}_+))\xrightarrow{\simeq}K(\mathcal{R}^S(M)). 
\]
\end{lemma}
\begin{proof}
The categories $\mathbb{C}_{fd}(M\times\mathbb{N}_+)$ and $\mathbb{D}_{fd}(M\times\mathbb{N}_+)$ are equipped with two notions of weak equivalences respectively. Applying Waldhausen's fibrations theorem, we can consider the following diagram of homotopy fibrations of \linebreak$\mathcal{S}$-constructions
\[
\begindc{0}[5]
\obj(9,10)[1]{$|h\mathcal{S}_\bullet\mathbb{C}_{fd}(M\times \mathbb{N}_+)^{h_S}|$}
\obj(32,10)[2]{$|h\mathcal{S}_\bullet\mathbb{C}_{fd}(M\times \mathbb{N}_+)|$}
\obj(53,10)[3]{$|h_S\mathcal{S}_\bullet\mathcal{R}^S(M)|$}
\obj(9,19)[4]{$|h\mathcal{S}_\bullet\mathbb{D}_{fd}(M\times \mathbb{N}_+)^{h_S}|$}
\obj(32,19)[5]{$|h\mathcal{S}_\bullet\mathbb{D}_{fd}(M\times \mathbb{N}_+)|$}
\obj(53,19)[6]{$|h_S\mathcal{S}_\bullet\mathbb{D}_{fd}(M\times \mathbb{N}_+)|$}
\mor{1}{2}{$ $}[\atleft,\solidarrow]
\mor{2}{3}{$ $}[\atleft,\solidarrow]
\mor{4}{5}{$ $}[0,0]
\mor{5}{6}{$ $}[\atleft,\solidarrow]
\mor{5}{2}{$ $}[1,0]
\mor{4}{1}{$ $}[0,0]
\mor{6}{3}{$ $}[0,0]
\enddc
\]

induced by the exact forgetful functor $F:\mathbb{D}_{fd}(M\times\mathbb{N}_+)\rightarrow \mathbb{C}_{fd}(M\times\mathbb{N}_+)$, mapping $(Y_-,Y,Y_+)$ to $Y_+$. By \cite[4.5]{FTI} the middle vertical arrow is a homotopy equivalence. Being induced by the restriction of $F$ the left vertical map can in the same way be shown to be a homtopy equivalence. This proves the lemma.
\end{proof}
\subsection{}The proof of the main theorem will be built on the following
\begin{proposition}\label{End-sequence}
There is a homotopy fibration sequence
\[
K(\End_{fd}^S(M))\rightarrow K(\Proj_{fd}(M))\xrightarrow{\lambda} K(\mathcal{R}_{fd}^S(M))
\]
\end{proposition}
\begin{proof}
Define a coarser class of weak equivalences on $\Proj_{fd}(M)$ by calling $\overline{f}:\overline{Y}\rightarrow\overline{Z}$ a weak equivalence if and only if $f_+:Y_+\rightarrow Z_+$ is a weak equivalence in $\mathcal{R}_{fd}^S(M)$ i.e. if $f_+$ is an $\tilde{S}$-equivalence. The corresponding Waldhausen category will be denoted by $h_{\NN_+}^{\tilde{S}}\Proj_{fd}(M)$. Applying Waldhausen's fibrations theorem gives a homotopy fibration sequence
\[
K(\Proj_{fd}^{h_{\NN_+}^{\tilde{S}}}(M))\rightarrow K(\Proj_{fd}(M))\rightarrow K(h_{\NN_+}^{\tilde{S}}\Proj_{fd}(M)) 
\]
We want to identify the third term with $K(\mathcal{R}_{fd}^S(M))$. By the preceding lemma, it is sufficient to show that the exact functor
\[
F:h_{\NN_+}^{\tilde{S}}\Proj_{fd}(M)\rightarrow h_S\mathbb{D}_{fd}(M\times\mathbb{N}_+)
\]
induced by inclusion, satisfies the approximation properties. Only the second property is of interest. Let $\bar{Y}$ be in $h_{\NN_+}^{\tilde{S}}\Proj_{fd}(M)$, $\bar{Z}$ in $h_S\mathbb{D}_{fd}(M\times\NN_+)$ and $\bar{f}:\bar{Y}\rightarrow\bar{Z}$ a morphism in $h_S\mathbb{D}_{fd}(M\times\NN_+)$. Consider the diagram

\[
\begindc{0}[5]
\obj(9,10)[1]{$Z$}
\obj(22,10)[2]{$Z$}
\obj(36,10)[3]{$Z_+$}
\obj(9,19)[4]{$Y_-$}
\obj(22,19)[5]{$Y$}
\obj(36,19)[6]{$Y_+$}
\mor{1}{2}{$=$}[\atleft,\solidarrow]
\mor{3}{2}{$ $}[\atleft,\solidarrow]
\mor{4}{5}{$a_-$}[0,0]
\mor{6}{5}{$ $}[\atleft,\solidarrow]
\mor{5}{2}{$f$}[1,0]
\mor{4}{1}{$f\circ a_-$}[1,0]
\mor{6}{3}{$f_+ $}[0,0]
\enddc
\]

and replace the vertical arrows by cofibrations
\[
\begindc{0}[5]
\obj(9,10)[1]{$\cyl(f)$}
\obj(22,10)[2]{$\cyl(f)$}
\obj(36,10)[3]{$\cyl(f_+)$}
\obj(9,19)[4]{$Y_-$}
\obj(22,19)[5]{$Y$}
\obj(36,19)[6]{$Y_+$}
\mor{1}{2}{$=$}[\atleft,\solidarrow]
\mor{3}{2}{$ $}[\atleft,\solidarrow]
\mor{4}{5}{$a_-$}[0,0]
\mor{6}{5}{$ $}[\atleft,\solidarrow]
\mor{5}{2}{$ $}[1,0]
\mor{4}{1}{$ $}[1,0]
\mor{6}{3}{$ $}[0,0]
\enddc
\]

Now put $\bar{Y}'=(\cyl(f),\cyl(f),\cyl(f_+))$, $\bar{Z}'=(Z,Z,Z_+)$ and $g:\bar{Y}\rightarrow\bar{Y}'$. The map $g$ is a cofibration in $h_{\NN_+}^{\tilde{S}}\Proj_{fd}(M)$ and the square

\[
\begindc{0}[5]
\obj(9,10)[1]{$F(\bar{Y}')$}
\obj(22,10)[2]{$\bar{Z}'$}
\obj(9,19)[4]{$F(\bar{Y})$}
\obj(22,19)[5]{$\bar{Z}$}
\mor{1}{2}{$\simeq_{h_S} $}[\atright,\solidarrow]
\mor{5}{2}{$\simeq_{h_S} $}[\atleft,\solidarrow]
\mor{4}{1}{$F(g)$}[0,0]
\mor{4}{5}{$ $}[0,0]
\enddc
\]

provides a factorisation of $\bar{f}$ up to weak equivalence. To apply Waldhausen's approximation theorem, this has to be made into a strict factorisation ($\bar{Z}'=\bar{Z}$). However since both Waldhausen categories involved are saturated and satisfy the cylinder axiom, we may apply the weaker version of the approximation theorem by Cisinski (cf. \cite[Prop. 2.14]{Cis10}) to deduce that $$K(F):K(h_{\NN_+}^{\tilde{S}}\Proj_{fd}(M))\rightarrow K(h_S\mathbb{D}_{fd}(M\times\mathbb{N}_+))$$ is a homotopy equivalence.

\end{proof}

\begin{lemma}\label{a_+ is equivalence}
 Let $\overline{Y}$ be in $\Proj_{fd}^{h^{\tilde{S}}_{\NN_+}}(M)$. The structure map $a_+$ is a stable homotopy equivalence, i.e. there exists an $N$ such that
 \[
\Sigma^Na_+:\Sigma^N Y_+\xrightarrow{\simeq}\Sigma^N Y 
 \]
 is a homotopy equivalence.
\end{lemma}
 \begin{proof}
Let $\tilde{g}$ be any polynomial with leading coefficient 1. The induced endomorphism $\tilde{g}(t)$ is mapped to the identity under the quotient map
\[
\{Y_+,Y_+\}\rightarrow \{Y_+/tY_+,Y_+/tY_+\}.
\]
Since $Y_+$ is stably annihilated by such a polynomial, the quotient $Y_+/tY_+$ is stably contractible.  Using the cofibration sequence
 \[
 tY_+\coRarrow Y_+\quotRarrow Y/tY_+
 \]
 we conclude that 
\[
 tY_+\rightarrow Y_+
\]
is a stable equivalence which implies that the action by $t$ is stably homotopy invertible. This implies that $Y_+$ is already stably homotopy equivalent to its telescope. 
 
\end{proof}

\begin{lemma}\label{Lemma:Third term is End}
 There is a natural homotopy equivalence of spectra
 \[
 K(\End^S_{fd}(M))\simeq K(\Proj^{h^{\tilde{S}}_{\NN_+}}(M)). 
 \]
 \end{lemma}
\begin{proof}
Let $(Y,f)$ be an object in $\End_{fd}^S(M)$ and consider the \textit{characteristic sequence} of \cite[p. 3029]{FTIII}, i.e. the homotopy coequaliser
\[
(\NN_-)_+\wedge Y\rightrightarrows(\NN_-)_+\wedge Y\rightarrow Y_f
\]
with the top arrow given by the shift endomorphism and the bottom one by $\id\wedge f$. The ``cofibre'' $Y_f$ lies in $\CC_{fd}(M\times \NN_-)$ with $t^{-1}$ acting via $f$. Furthermore, since essentially per construction it is $S$-contractible there exists a $g$ stably homotopy annihilating $t^{-1}$. We observe that the action by $t^{-1}$ on the telescope $Y_f(t)$ is also stably homotopy annihilated by $g$. On the telescope, however, the action is invertible and because of the relation 
\[
g(t^{-1})=\tilde{g}(t)\cdot t^n
\]
$\tilde{g}$ stably homotopy annihilates the $t$-action on $Y_f(t)$. Hence we may view $Y_f(t)$ with the action by $t$ as an object in $\mathcal{R}_{fd}^S(M)$. Using the the preceding lemma one can now check that the correspondence
\[
(Y,f)\mapsto (Y_f,Y_f(t),Y_f(t)) 
\]
defines an exact functor $\Phi:\End^S_{fd}(M)\rightarrow\Proj^{h^{\tilde{S}}_{\NN_+}}(M)$.

Conversely consider now
\[
\overline{Y}=(Y_-,Y,Y_+)\in\Proj^{h^{\tilde{S}}_{\NN_+}}(M).  
\]
Since $Y_+$ is $\tilde{S}$-contractible there is a $\tilde{g}$ such that 
\[
\hocolim(\cdots\xrightarrow{\tilde{g}(t)}Y_+\xrightarrow{\tilde{g}(t)}Y_+\xrightarrow{\tilde{g}(t)}\cdots)
\]
is contractible and thus also 
\[
\hocolim(\cdots\xrightarrow{\tilde{g}(t)}Y\xrightarrow{\tilde{g}(t)}Y\xrightarrow{\tilde{g}(t)}\cdots)
\]
since $Y\simeq Y_+(t^{-1})$ and $\tilde{g}(t)$ lifts to a map of diagrams. As before we may then conclude that the reverse polynomial $g$ of $\tilde{g}$ stably homotopy annihilates the $t^{-1}$-action on $Y_-(t)\simeq Y(t)\cong Y$, where $g(x^{-1})=\tilde{g}(x)x^n$, i.e. that 
\[
\hocolim(\cdots\xrightarrow{g(t^{-1})}Y_-(t)\xrightarrow{g(t^{-1})}Y_-(t)\xrightarrow{g(t^{-1})}\cdots)\simeq\ast.
\]

Since $g(x^{-1})$ is not divisible by $x^{-1}$, the latter is equivalent to
\[
\hocolim(\cdots\xrightarrow{t^{-1}g(t^{-1})}Y_-\xrightarrow{t^{-1}g(t^{-1})}Y_-\xrightarrow{t^{-1}g(t^{-1})}\cdots) 
\]
and its contractibility implies that there exists a $k\geq 0$ such that
\[
t^{-k}g^k(t^{-1}):Y_-\rightarrow Y_- 
\]
is $M$-nullhomotopic.

Putting this together shows that $g'(t^{-1}):=t^{-k}g^{k}(t^{-1})$ vanishes in $\{Y_-,Y_-\}$ and $(Y_-,t^{-1})$ is an object of $\End^S(M)$ \textit{without} any finiteness conditions. We claim that $(Y_-,t^{-1})$ is in fact stably finitely dominated which is sufficient for the purposes of $K$-theory.
Consider the homotopy pullback of the following diagram taken in the category $\mathbb{C}(M)$
\[
\begindc{0}[40]
\obj(1,1)[1]{$Y_-$} 
\obj(2,1)[2]{$Y$}
\obj(2,2)[3]{$Y_+$}
\mor{1}{2}{$a_-$}[-1,0]
\mor{3}{2}{$a_+$}[1,0]
\enddc 
\]

which we denote by $D$. Since we are working stably $\Sigma D$ is homotopy equivalent (in $\mathbb{C}(M)$) to the homotopy cofibre of $Y_-\vee Y_+\rightarrow Y$ which in turn is equivalent to the global sections $\Gamma(\overline{Y})$ of $\overline{Y}=(Y_-\xrightarrow{a_-}Y\xleftarrow{a_+}Y_+)$ (cf. \cite[5.1]{FTI}). Since $\overline{Y}$ is finitely dominated and $\Gamma$ maps finitely dominated objects to stably finitely dominated ones we conclude that $D$ is stably finitely dominated. By \ref{a_+ is equivalence}, $a_+$ is a stable equivalence and the homotopy pullback square
\[
\begindc{0}[50]
\obj(1,1)[1]{$Y_-$} 
\obj(2,1)[2]{$Y$}
\obj(2,2)[3]{$Y_+$}
\obj(1,2)[4]{$D$}
\mor{1}{2}{$a_-$}[-1,0]
\mor{4}{1}{$\simeq$}[-1,0]
\mor{3}{2}{$a_+$}[1,0]
\mor{3}{2}{$\simeq$}[-1,0]
\mor{4}{3}{$ $}
\enddc 
\]

shows the claim. The assignment
\[
(Y_-,Y,Y_+)\mapsto (Y_-,t^{-1}) 
\]
defines an exact functor $\Psi:\Proj^{h^{\tilde{S}}_{\NN_+}}(M)\rightarrow \End^S_{sfd}(M)$. The composition $\Psi\circ\Phi$ sends $(Y,f)$ to $(Y_f,f)$.  
Since by \cite[4.1]{FTIII} $Y_f$ is homology equivalent to $Y$ we have $\Psi\circ\Phi\simeq\id$. Because of \ref{a_+ is equivalence} the composition $\Phi\circ\Psi$ is also equivalent to $\id$.
\end{proof}
\subsection{}
\begin{proof}[Proof of \ref{Main Theorem}]
According to Prop. \ref{Proposition: Decomposition of K-theory of the projective line}  we have a fibration
\[
K(\CC_{fd}(M))\xrightarrow{\phantom{i}\psi_{-1}\phantom{i}} K(\Proj_{fd}(M))\xrightarrow{\phantom{\psi}\Gamma\phantom{\psi}} K(\CC_{fd}(M)).
\]
Combining this with the sequence provided by \ref{End-sequence} we look at the diagram
 \[
\begindc{0}[5]
\obj(9,10)[1]{$\phantom{n}K(\CC_{fd}(M))\phantom{d}$}
\obj(28,10)[2]{$K(\Proj_{fd}(M))$}
\obj(46,10)[3]{$K(\CC_{fd}(M))$}
\obj(9,19)[4]{$\tilde{K}(\End_{fd}^S(M))$}
\obj(28,19)[5]{$K(\End_{fd}^S(M))$}
\obj(46,19)[6]{$K(\CC_{fd}(M))$}
\obj(28,28)[7]{$\Omega K(\mathcal{R}_{fd}^S(M))$}

\mor{1}{2}{$\psi_{-1}$}[\atleft,\solidarrow]
\mor{2}{3}{$\Gamma $}[\atleft,\solidarrow]
\mor{4}{5}{$ $}[0,1]
\mor{5}{6}{$\mathrm{p}$}[\atleft,\solidarrow]
\mor{5}{2}{$\Phi$}[1,0]
\mor{7}{5}{$ $}[0,0]
\enddc
\]
 
We would like to complete the left hand square to a pullback square. For this we claim that the dashed map is the inclusion of the fibre i.e. that
\[
\tilde{K}(\End_{fd}^S(M))\xrightarrow{\phantom{\Gamma\circ\Phi}} K(\End_{fd}^S(M))\xrightarrow{\Gamma\circ\Phi} K(\CC_{fd}(M))
\] 
is a homotopy fibration sequence. The projection functor $\mathrm{p}:\End_{fd}^S(M)\rightarrow\CC_{fd}(M)$ has a natural splitting functor $\mathrm{j}$ given by sending $Y$ to $(Y,\ast)$. We are going to show that $\mathrm{j}$ also provides a (stable) splitting for $\Gamma\circ\Phi$ which will imply the claim.\linebreak

\noindent The space $Y$ is mapped by $\Phi\circ \mathrm{j}$ to the triple 
\[
(Y_\ast,Y_\ast(t),Y_\ast(t)) 
\]
where as in the proof of \ref{Lemma:Third term is End} 
$Y_\ast$ is given by the homotopy coequalizer sequence\linebreak
\[
(\mathbb{N}_-)_+\wedge Y\rightrightarrows(\mathbb{N}_-)_+\wedge Y\rightarrow Y_\ast
\]
of the shift action and the summandwise trivial map $\ast$. This is mapped via $\Gamma$ to\linebreak
\[
cY_\ast\cup_{Y_{\ast}}Y_{\ast}(t)  \cup_{Y_{\ast}(t)}
cY_{\ast}(t)
\]
which contracts to one copy of $\Sigma Y_\ast$. In the standard model for the coequaliser given by the double mapping cylinder, each cylinder summand on $Y$ can be collapsed to the image of $\ast(Y)=\ast$. This provides a natural homotopy equivalence from $Y_\ast$ to $Y$ and thus from $\Sigma Y_{\ast}$ to $\Sigma Y$. Since suspension induces an equivalence on $K$-theory this shows $\Gamma\circ\Phi\circ \mathrm{j}\simeq \id $.\\

\noindent Thus by the claim we end up with a commutative diagram of fibration sequences

 \[
\begindc{0}[5]
\obj(9,10)[1]{$K(\CC_{fd}(M))$}
\obj(28,10)[2]{$K(\Proj_{fd}(M))$}
\obj(46,10)[3]{$K(\CC_{fd}(M))$}
\obj(9,19)[4]{$\tilde{K}(\End_{fd}^S(M))$}
\obj(28,19)[5]{$K(\End_{fd}^S(M))$}
\obj(46,19)[6]{$K(\CC_{fd}(M))$}

\mor{1}{2}{$\psi_{-1}$}[\atleft,\solidarrow]
\mor{2}{3}{$\Gamma $}[\atleft,\solidarrow]
\mor{4}{5}{$ $}[0,0]
\mor{5}{6}{$\Gamma\circ\Phi $}[\atleft,\solidarrow]
\mor{5}{2}{$\Phi$}[1,0]
\mor{4}{1}{$ $}[0,0]
\mor{6}{3}{$\id$}[0,0]
\enddc
\]

and consequently with a homotopy (co)cartesian square
\[
\begindc{0}[5]
\obj(10,10)[1]{$K(\CC_{fd}(M))$}
\obj(28,10)[2]{$K(\Proj_{fd}(M))$}
\obj(10,19)[4]{$\tilde{K}(\End_{fd}^S(M))$}
\obj(28,19)[5]{$K(\End_{fd}^S(M))$}

\mor{1}{2}{$\psi_{-1}$}[\atleft,\solidarrow]
\mor{4}{5}{$ $}[0,0]
\mor{5}{2}{$\Phi$}[1,0]
\mor{4}{1}{$ $}[0,0]
\enddc
\]


which provides a fibration 
\[
\tilde{K}(\End_{fd}^S(M))\xrightarrow{\qquad} K(\CC_{fd}(M))\xrightarrow{\lambda\circ\psi_{-1}} K(\mathcal{R}_{fd}^S(M))
\]
with $\lambda$ the localisation of \ref{End-sequence}
\[
\lambda:K(\Proj_{fd}(M))\hookrightarrow K(h^{\tilde{S}}_{\mathbb{N}_+}\Proj_{fd}(M))\xrightarrow{\;\simeq\;}K(\mathcal{R}_{fd}^S(M)).
\]
The projection functor $\mathrm{p}:\mathcal{R}_{fd}^S(M)\rightarrow\CC_{fd}(M)$ maps an $M\times\mathbb{N}$-space to its $\NN$-orbit and can be identified with the augmentation functor $\varepsilon$ of\linebreak \cite[7.1]{FTI}. Applying \cite[7.2]{FTI} shows that $\mathrm{p}$ provides a splitting for $\lambda\circ\psi_{-1}$. Consequently the above fibration can be extended to the right
\[
\tilde{K}(\End_{fd}^S(M))\xrightarrow{\qquad} K(\CC_{fd}(M))\xrightarrow{\lambda\circ\psi_{-1}} K(\mathcal{R}_{fd}^S(M))\xrightarrow{\qquad}E^S_{fd}(M)
\]
with $E^S_{fd}(M)=\tilde{K}(\mathcal{R}_{fd}^S(M))$. This finishes the proof.
\end{proof}

\bibliographystyle{amsalpha}
\bibliography{Bibliography}
\end{document}